\documentclass[a4paper,12pt]{amsart}
\usepackage{amsmath,amssymb,amsthm}
\usepackage[left=2cm,right=2cm]{geometry}
\makeatletter
\@namedef{subjclassname@2020}{%
	\textup{2020} Mathematics Subject Classification}
\makeatother

\newtheorem{thm}{Theorem}
\newtheorem{res}{Result}
\newtheorem{prop}[thm]{Proposition}
\newtheorem{lem}[thm]{Lemma}
\newtheorem{cor}[thm]{Corollary}

\theoremstyle{definition}
\newtheorem{defn}{Definition}
\theoremstyle{remark}
\newtheorem{rem}{Remark}
\def\C{\mathbb C}

\def\hol{\mathcal O}

\def\O{\Omega}

\def\k{\kappa}

\def\ov{\overline}

\def\lk{ l^{\kappa}}
\def\Cn{\mathbb{C}^n}

\renewcommand{\Im}{\operatorname{\rm{Im}}}
\renewcommand{\Re}{\operatorname{\rm{Re}}}

\begin{document}
	
	\title[On non-visibility of Kobayashi geodesics and the geometry of the boundary]
	{On non-visibility of Kobayashi geodesics and the geometry of the boundary}
	
	\author{Ahmed Yekta {\"O}kten}
	
	\address{A.Y.Ökten\\
		Institut de Math\'ematiques de Toulouse; UMR5219 \\
		Universit\'e de Toulouse; CNRS \\
		UPS, F-31062 Toulouse Cedex 9, France} \email{ahmed$\_$yekta.okten@math.univ-toulouse.fr}

\begin{abstract}  We show that on convex domains with $\mathcal C^{1,\alpha}$-smooth boundary the limit set of non-visible Kobayashi geodesics are contained in a complex face. In $\C^2$ this implies the existence of a complex tangential line segment of non-Goldilocks point in the boundary. Conversely, we construct sequences of non-visible almost-geodesics on ($\C-$)convex domains whose boundary contains a complex tangential line segment of non-Goldilocks points in a specific direction. 
\end{abstract}

\thanks{The author received support from the University Research School EUR-MINT (State support managed
	by the National Research Agency for Future Investments program bearing the reference ANR-18-EURE-0023).}



\maketitle
\section{Introduction}

The Kobayashi distance is a fundamental holomorphic invariant. The study of its metric geometric properties turned out to be an effective tool; for instance, Gromov hyperbolicity of the Kobayashi distance is exploited to obtain results about iteration theory and boundary extension of holomorphic maps. Zimmer \cite{Z2} showed that a smoothly bounded convex domain is Gromov hyperbolic if and only if it is of finite D'Angelo type, establishing a connection between the metric geometry of the Kobayashi distance and the Euclidean geometry of the boundary. 

However, for many applications, the somewhat weaker property of visibility is sufficient. Informally speaking, it means that all geodesics joining points tending to two distinct given boundary points have to pass through the same compact set, staying away from the boundary. Bharali-Zimmer \cite{BZ} defined a class of domains called \emph{Goldilocks domains} in terms of the growth of the Kobayashi distance and the Kobayashi-Royden pseudometric, and they showed that such domains satisfy visibility property for almost-geodesics, i.e. the curves that approximate Kobayashi geodesics in a suitable sense. The paper \cite{BZ2} defined the notion of \emph{local Goldilocks points} in a similar manner and showed that local Goldilocks points cannot be in the limit set of almost-geodesics tending to the boundary.

In the case of convex domains, local Goldilocks points can be described in terms of the "flatness" of the boundary. In particular finite type boundary points are local Goldilocks points, yet local Goldilocks point conditions allow infinite type to some extent. On the other hand, \cite[Proposition 5.4]{BNT} provides an example of a convex domain with a complex tangential line segment of infinite type points which fails the visibility property. 

The goal of this note is to present some sufficient and necessary conditions for the failure of visibility on convex domains in terms of the Euclidean geometry of the boundary. Notably, we prove Theorem \ref{thm:convexface} which shows that on convex domains with $\mathcal{C}^{1,\alpha}$-smooth boundary, Kobayashi geodesics (and also certain almost-geodesics) tending to the boundary must tend to a complex hyperplane. In particular, in $\C^2$ (see Corollary \ref{cor:nonvisibilityinc2}) this implies that the boundary is flat in the precise sense that it contains a complex tangential line segment of non-Goldilocks points. The proof of this result follows from Proposition \ref{prop:commongromovpartner}, where we improve several results in the literature that relate visibility to the growth of Kobayashi distance.

Conversely, inspired by \cite[Proposition 5.4]{BNT} we prove Theorem \ref{thm:nonvisibleinC2}, whose proof shows that convex domains in $\C^2$ with a boundary containing a complex tangential line segment of points with a strong enough flatness (so in particular not Goldilocks) possess almost-geodesics that tend to the boundary. Consequently, we observe that the local Goldilocks point conditions are somewhat necessary for visibility, providing a partial converse to \cite[Theorem 1.4]{BZ}. Theorem \ref{thm:nonvisibleinC2} extends to higher dimensions and to $\C$-convex domains under additional assumptions (see Corollaries \ref{cor:corollaryforCconvex} and \ref{cor:howtheresultextendtohigherdimensions}). Notably, the domains given in Corollary \ref{cor:howtheresultextendtohigherdimensions} are not Gromov hyperbolic with respect to the Kobayashi distance, and this improves \cite[Theorem 1.5]{Z2}.

\section{Preliminaries}

Let $\Omega$ be a domain in $\C^n$, $z,w\in \Omega$ and $v\in\Cn$.
Recall that the Kobayashi pseudodistance $k_\Omega$ is the largest pseudodistance
that does not exceed the Lempert function
$$ l_\O(z,w):=\tanh^{-1} \tilde{l}_\O(z,w),$$
where $\Delta$ is the
unit disc and $\tilde{l}_\Omega(z,w):=\inf\{|\alpha|:\exists\varphi\in\hol(\Delta,\O)
\hbox{ with }\varphi(0)=z,\varphi(\alpha)=w\} $. Explicitly, the Kobayashi length of any continuous curve $\gamma:[0,1]\rightarrow\Omega$ is defined as 
$$\displaystyle{l^k_\O(\gamma):=\inf_{n\in\mathbb{N}} \left\{ \sum_{j=1}^n l_\Omega(z_{j-1},z_j):z_0=\gamma(0), z_n=\gamma(1), \: z_j\in\gamma([0,1])\right\}} $$ and
$k_\Omega(z,w):=\displaystyle{ \inf l^k_\O(\gamma) },$ where the infimum is taken over all continuous curves joining $z,w\in\Omega$. A continuous curve attaining the infimum above is said to be a Kobayashi geodesic. 

The Kobayashi-Royden pseudometric is given by
$$
\k_\O(z;v)=\inf\{1/r:r > 0,\exists\varphi\in\mathcal{O}(r\Delta,\O),\varphi(0)=z,\varphi'(0)=v\}.
$$

By \cite[Theorem 1.2]{V}, one may express $k_\Omega$ as the inner distance associated with the Kobayashi-Royden pseudometric. That is 
$ k_\O(z,w)=\inf\lk_\O(\gamma),$ where $$\lk_\O(\gamma):=\int_{I} \k_\O(\gamma(t),\gamma'(t))dt$$ and the infimum is taken over all absolutely continuous curves joining $z$ to $w$. 

We say that a domain $\O\subset \Cn$ is \textit{hyperbolic} if $k_\O$ is a distance and \textit{complete hyperbolic} if $(\O,k_\O)$ is complete as a metric space. For instance, any convex domain containing no affine complex lines is complete hyperbolic. The Hopf-Rinow theorem \cite[Proposition 3.7]{BH} asserts that if $\O$ is complete hyperbolic then any pair $z,w\in\O$ is joined by a Kobayashi geodesic. In general, geodesics for the Kobayashi distance need not exist. However, the definition of the Kobayashi distance as an inner distance ensures the existence of what is termed "almost geodesics" in \cite[Definition 4.1]{BZ}. In this note, we will call them $(\lambda,\epsilon)$-geodesics:

\begin{defn}
	Let $\O$ be a domain in $\Cn$, $\gamma:[0,1]\rightarrow\O$ be an absolutely continuous curve, $\lambda\geq 1$ and $\epsilon\geq 0$. We say that $\gamma$ is a \emph{$(\lambda,\epsilon)$-geodesic} if for all $t_1 \leq t_2 \in [0,1]$ we have that $$ \lk_\O(\gamma|_{[t_1,t_2]}) \leq \lambda k_\O(\gamma(t_1),\gamma(t_2))+\epsilon.$$
\end{defn}

For brevity, throughout this note $(1,\epsilon)$-geodesics will be called "$\epsilon$-geodesics". 


Since the geodesics for the Kobayashi-Royden pseudometric need not exist, Bharali-Zimmer \cite{BZ} defined a notion of visibility in terms of $(\lambda,\epsilon)$-geodesics (see also the remark below \cite[Definition 1.1]{CMS}). 

\begin{defn}\label{defn:bzvisibility}
	Let $\O$ be a hyperbolic domain, $\lambda\geq 1$ and $\epsilon\geq 0$. We say that the pair of distinct points $\{p,q\}\subset\partial\O$ is a $(\lambda,\epsilon)$-visible pair if there exist neighbourhoods $U_p, U_q$ of $p,q$ respectively and $K\subset\subset\O$ such that if $\gamma:I\rightarrow\O$ is a $(\lambda,\epsilon)$-geodesic joining a point $z\in\O\cap U_p$ to a point $w\in\O\cap U_q$ then $\gamma(I)\cap K \neq \emptyset$. We say that $\O$ satisfies $(\lambda,\epsilon)$-visibility property if any pair of distinct points $\{p,q\}\subset\O$ is a $(\lambda,\epsilon)$-visible pair. We say that $\O$ satisfies the \textit{visibility property} (respectively the \textit{weak visibility property}) if it satisfies the $(\lambda,\epsilon)$-visibility property for any $\lambda \geq 1, \epsilon\geq 0$ (respectively $\lambda=1$ and $\epsilon\geq 0$).
\end{defn} 

Similarly, if $\O$ is complete hyperbolic we say that a pair of distinct points $\{p,q\}\subset\partial\O$ is a \emph{geodesic visible pair} if there exist neighbourhoods $U_p, U_q$ of $p,q$ respectively and $K\subset\subset\O$ such that if $\gamma:I\rightarrow\O$ is a Kobayashi geodesic joining a point $z\in\O\cap U_p$ to a point $w\in\O\cap U_q$ then $\gamma(I)\cap K \neq \emptyset$. We say that $\O$ satisfies the \emph{geodesic visibility property} if any pair of distinct points $\{p,q\}\subset\partial\O$ is a geodesic visible pair.





Recall the definition of the Gromov product with respect to Kobayashi distance: 
$$ (z|w)_o^\O:=\dfrac{1}{2}\left(k_\O(z,o)+k_\O(w,o)-k_\O(z,w)\right), \:\:\:\: z,w\in\O.$$

By \cite[Proposition 2.4]{BNT} (see also \cite[Proposition 3.1]{CMS}) the visibility property is related to the growth of Kobayashi distance:

\begin{prop}\label{prop:growthbnt} 
	Let $\O$ be a complete hyperbolic domain. Then the following are equivalent. \begin{enumerate}
		\item $\O$ satisfies the geodesic visibility property.
		\item $\O$ satisfies the weak visibility property.
		\item 	\begin{equation}\label{eqngrowthbnt}
\forall p,q\in\partial \Omega \:\:\:\text{with}\:\:\:p\neq q, \:\:\:			\limsup_{\O\times\O \ni (z,w)\rightarrow (p,q)} (z|w)_o^\O < \infty.
		\end{equation} 
	\end{enumerate}
\end{prop}

In Section \ref{sec:visibilityandthegromovproduct} we will give a local version of Proposition \ref{prop:growthbnt}.

\section{Failure of visibility on complete hyperbolic domains}\label{sec:visibilityandthegromovproduct}

\subsection{Limits of non-visible geodesics and the Gromov product}

Let $\O$ be a domain in $\Cn$ and set
$$ \delta_\O(z):=\inf\{\|z-w\|:w\in\partial\O\} \:\:\:\text{and}\:\:\:\delta_\O(z;v):=\sup\{r>0:z+\alpha v \in \O \:\:\: \text{for all} \:\:\: \alpha \in r \Delta\}. $$ 

Suppose that $\gamma_n: [0,1]\rightarrow \O$ is a sequence of curves satisfying $$ \lim_{n\rightarrow\infty}\max\{\delta_\Omega(\gamma_n(t)):t\in [0,1]\} = 0,$$ $\lim_{n\to\infty} \gamma_n(0) = p$, $\lim_{n\to\infty} \gamma_n(1) =q$ with $q\neq p$. We denote their limit set as 
\begin{equation}\label{eqn:limitset}
	L:=\{p\in\partial\O: \text{$p$ is a subsequential limit of a sequence $\{x_n\}$ such that $\forall n\in\mathbb N$, $x_n\in\gamma_n([0,1])$}\}.
\end{equation}

\begin{prop}\label{prop:limitsetiscontinium}
	$L\subset\partial\O$ is closed and connected, therefore it is a continuum.
\end{prop}

\begin{proof}
	It is clear that $L$ is a closed set as it is defined to be a limit set. 
	
	Note that $L$ is connected if and only if it cannot be written as the union of two non-empty open seperated sets, that is for any two non-empty open sets $U, V\subset L$ such that $L=U\cup V $ we have $\overline{U}\cap V \neq \emptyset$. Let now $U_1,U_2$ be two open sets in $\mathbb{C}^n$ such that $L\subset U_1\cup U_2$. Set $V_1:=U_1\cap L$ and $V_2:=U_2\cap L$, we will show that $\overline{V_1}\cap V_2 \neq \emptyset$. We may assume that $p:=\lim_{n\to\infty}\gamma_n(0) \in U_1\setminus U_2$ and $q:=\lim_{n\to\infty}\gamma_n(0) \in U_2\setminus U_1$ otherwise $V_1\cap V_2\neq \emptyset$. 
	
	We will show that $\overline V_1 \cap V_2 \neq \emptyset$. Since each $\gamma_n([0,1])$ is connected and $p\in U_1$ and $q \notin U_1$ we may find $x_n \in \gamma_n([0,1])$ such that $x_n \in\partial U_1$. Since the images of $\gamma_n$ tend to $\partial \Omega$, by looking at a subsequence if necessary we may set $r:=\lim_{n\to\infty} x_n$. Then $r \in \partial{U_1}\cap L$. Since $L=V_1\cup V_2$ and $r \notin V_1$ it follows that $r \in V_2$. As $\overline V_1= \overline U_1 \cap L$ it follows that $\overline V_1 \cap V_2 \neq \emptyset$. As $U_1$ and $U_2$ are choosen arbitrarily the connectivity of $L$ follows.  
	
	
\end{proof}

\begin{defn}
	Let $\O\subset\Cn$ be a complete hyperbolic domain, $p\in\partial\O$. We say that $q\in\partial\O$ is a \textit{Gromov partner} of $p$ if $$\displaystyle{\limsup_{\O\times\O \ni (z,w)\to (p,q)}(z|w)^\O_o=\infty}.$$
	
	We denote the set of Gromov partners of $p$ by $G_p$.
\end{defn}

It is clear that the above definition is independent from the base point $o$, and that for any $p\in\partial\O$ we have $p\in G_p$.

The main result of this section is the following:
\begin{prop}\label{prop:commongromovpartner}
	Let $\O\subset\Cn$ be a complete hyperbolic domain, and suppose that $\{p,q\}\in\partial\O$ is not a weakly visible pair. Let $\gamma_n:[0,1]\to\O$ be a sequence of $\epsilon$-geodesics such that $\max_{t\in[0,1]}\delta_\O(\gamma_n(t))\to 0$, $\gamma_n(0)\to p$ and $\gamma_n(1)\to q$. Let $L$ be as in \eqref{eqn:limitset}. \begin{enumerate} 
		\item There exists $r\in L\setminus\{p,q\}$ with $r\in G_p\cap G_q$.
		\item $L\subset G_p\cup G_q$.
	\end{enumerate}.
\end{prop}
\kern-2em

\begin{proof} 
	
		Let $\gamma_n:[0,1]\rightarrow\O$ be as above, set $z_n:=\gamma_n(0)$ and $w_n:=\gamma_n(1)$. Fix $o\in\O$ and set $h_n:[0,1]\rightarrow \O$ to be $h_n(t):=2\left((z_n|\gamma_n(t))^\O_o - (w_n|\gamma_n(t))^\O_o\right)$.  Then by a direct calculation, we have $ h_n(0)= 2 k_\O(z_n,o)-(k_\O(z_n,o)+k_\O(w_n,o)-k_\O(z_n,w_n))$ and $h_n(1)= (k_\O(z_n,o)+k_\O(w_n,o)-k_\O(z_n,w_n))-2 k_\O(w_n,o)$ so by the triangle inequality $$h_n(0)=2(w_n|o)_{z_n}^\O \geq 0 \:\:\:\:\: \text{and} \:\:\:\:\: h_n(1)=-2(z_n|o)_{w_n}^\O\leq 0.$$ 
		
		Note that $\O$ is complete hyperbolic hence it is taut. This shows that $\kappa_\O$ is continuous on $\O\times \Cn$ and hence $k_\O$ is continuous on $\O$. In particular, this implies that each $h_n$ is a continuous function. An application of the intermediate value theorem to $h_n$ shows that there exists a $\tau_n\in [0,1]$ such that $h_n(\tau_n)=0$, that is $(z_n|\gamma_n(\tau_n))^\O_o=(w_n|\gamma_n(\tau_n))^\O_o$. Set $x_n:=\gamma_n(\tau_n)$. Then a direct calculation shows that $$2(z_n|x_n)^\O_o=2(w_n|x_n)^\O_o=(z_n|x_n)^\O_o+(w_n|x_n)^\O_o \geq (z_n|w_n)_o^\O + k_\O(x_n,o) -\epsilon/2 \geq k_\O(x_n,o) -\epsilon/2.$$ As $\O$ is complete hyperbolic  $k_\O(x_n,o)$ tends to infinity, so both $(z_n|x_n)^\O_o,(w_n|x_n)^\O_o$ tend to infinity. By passing to a subsequence, we observe that $\lim_{n\rightarrow\infty} x_n =: r \in G_p\cap G_q$. If $r\notin \{p,q\}$ we are done. 
	
	If not, without loss of generality assume that $r=p$. Take $\tau'_n\in[\tau_n,1]$ such that $x'_n:=\gamma_n(\tau'_n)$ verifies $\|x'_n-w_n\|=\|x'_n-x_n\|$. Passing to a subsequence if necessary set $r'= \lim_{n\rightarrow\infty} x'_n$. Clearly, $r'\notin\{p,q\}$. As $x'_n\in\gamma_n([\tau_n,1])$ we have
	\begin{multline*}
		(x'_n|w_n)^\O_o - (x_n|w_n)^\O_o = \dfrac{1}{2}[k_\O(x'_n,o)-k_\O(x_n,o)+k_\O(x_n,w_n)-k_\O(x'_n,w_n)] \geq \\
		\dfrac{1}{2}[k_\O(x'_n,o)-k_\O(x_n,o)+\lk_\O(\gamma_n|_{[c_n,b_n]})-\lk_\O(\gamma_n|_{[c'_n,b_n]})-\epsilon] \geq \dfrac{1}{2}[k_\O(x'_n,o)-k_\O(x_n,o)+k_\O(x'_n,x_n)-\epsilon].
	\end{multline*}
	So by the triangle inequality we get
	$(x'_n|w_n)^\O_o \geq (x_n|w_n)^\O_o-\epsilon/2$. Similarly, $(x'_n|x_n)^\O_o \geq (x_n|w_n)^\O_o-\epsilon/2$. By construction $(x_n|w_n)^\O_o$ tends to infinity; hence so do $(x'_n|w_n)^\O_o$ and $(x'_n|x_n)^\O_o$. As $x_n\rightarrow p$, $x'_n\rightarrow r'$ and $w_n\rightarrow q$ we have $r'\in G_p \cap G_q$. We have the first statement.
	
	To see the second statement, by the first statement assume that there exists $\tau_n \in [0,1]$ with $(\gamma_n(\tau_n)|z_n)^\O_o$ and $(\gamma_n(\tau_n)|w_n)^\O_o$ tending to infinity. Suppose that $s\in L$, then there exists $\tau'_n\in [0,1]$ with $s_n:=\gamma_n(\tau'_n)$ tending to $s$. Switching $p$ and $q$ if needed, by passing to a subsequence we may assume that $\tau_n \leq \tau'_n$. As $s_n\in\gamma_n([\tau_n,1])$ similar calculations as above show that $(s_n|w_n)^\O_o\geq (\gamma_n(\tau_n)|w_n)^\O_o - \epsilon/2$. Thus, $(s_n|w_n)^\O_o$ tends to infinity hence $s\in G_q$. As $s\in L$ is arbitrary this finishes the proof. 
\end{proof}

Unfortunately, it is not clear whether in the above scenario we have that $q\in G_p$. 




The proof of Proposition \ref{prop:commongromovpartner} leads to the following corollary which improves the conclusion of \cite[Proposition 2.5]{BNT} and \cite[Proposition 3.1]{CMS}. 

\begin{cor}\label{cor:limitsetandgromovproduct}
	Let $\O$ be a complete hyperbolic domain and suppose that $p\in\partial\O$ is not a weakly visible point. Then, there exists a $\delta>0$ such that whenever $\delta'\leq \delta$, there exists $r\in G_p$ such that $\|p-r\|=\delta'$. 
\end{cor}
\begin{proof}
	Suppose that $p$ is not a weakly visible point, then there exists $q\in\partial\O$ different from $p$ so that $\{p,q\}$ is not a weakly visible pair. Let $\gamma_n:[0,1]\rightarrow\O$ be a sequence of $\epsilon$-geodesics joining $\gamma_n(0)=:z_n$ to $\gamma_n(1)=:w_n$ such that $\lim_{n\rightarrow\infty}\max\{\delta_\Omega(\gamma_n(t)):t\in [0,1]\} = 0$ and  $z_n\to p$, $w_n\to q$. Let $L$ be as in \eqref{eqn:limitset}. By Proposition \ref{prop:commongromovpartner} there exists $r\in L$, different from $p$ such that $r\in G_p$. Then, there exists $x_n\in \gamma_n([0,1])$ tending to $r$ and $(z_n|x_n)_o^\O$ tending to infinity. Set $\delta:=\|p-r\|$ and let $\delta'\in (0,\delta)$ be arbitrary. By connectivity of geodesics, looking at the tail of the sequence if necessary we may choose $x'_n$ in the part of $\gamma_n(I_n)$ joining $z_n$ to $x_n$ with $\|z_n-x'_n\|=\delta'$. The proof of Proposition \ref{prop:commongromovpartner} shows that $(x'_n|z_n)_o^\O\geq (z_n|x_n)_o^\O -\epsilon/2$, thus $(x'_n|z_n)_o^\O$ tends to infinity. Looking at a subsequence set $\lim_{n\to \infty} x'_n=:r' \in \partial\O$. Then $\|r'-p\|=\delta'$ and $r'\in G_p$.  \end{proof}

\begin{rem}\label{rem:remarkabouttheargumentsworkforgeodesics}
	Taking $\epsilon=0$ and replacing the Kobayashi-Royden length with the Kobayashi length, the proof of Proposition \ref{prop:commongromovpartner} shows that it also remains true when we consider the case where $\{p,q\}$ is not a geodesic visible pair and we consider the limit set of Kobayashi geodesics tending to the boundary. Similarly, Corollary \ref{cor:limitsetandgromovproduct} remains true when $p$ is not a geodesic visible point.   
\end{rem}

\begin{rem}
	Observe that the proof of Corollary \ref{cor:limitsetandgromovproduct} shows that the set of Gromov partners of $p\in\partial\O$ can be realized as a limit set of curves tending to the boundary. In particular, by Proposition \ref{prop:limitsetiscontinium} we see that $G_p\subset\partial\O$ is a continuum.  
\end{rem}
\subsection{Applications of Proposition \ref{prop:commongromovpartner}}

The goal of this subsection is to give some applications of Proposition \ref{prop:commongromovpartner}. In particular, we present a local version of Proposition \ref{prop:growthbnt} and we provide a short proof of the main result of \cite{NOT} in the case of complete hyperbolicity. 

Let $\O$ be a domain in $\Cn$. Following \cite[Proposition 8]{NOT} we say that $p\in\partial\O$ is a \textit{visible (resp. weakly visible, resp. geodesic visible) point} if and only if for any $q\in\partial\O$ with $q\neq p$ we have that $\{p,q\}$ is a visible (resp. weakly visible, resp. geodesic visible) pair. It turns out that being a visible point is a local property independent of the domain itself, see \cite[Theorem 15]{NOT}.

With this definition, the following is a local version of Proposition \ref{prop:growthbnt}. The difference is that we make use of Proposition \ref{prop:commongromovpartner} instead of \cite[Proposition 2.4]{BNT}.

\begin{cor}\label{cor:characterizationofvisibility}
	Let $\O$ be a complete hyperbolic domain, $p\in\partial\O$. The following are equivalent. 
	\begin{enumerate}
		\item $p$ is not a geodesic visible point.
		\item $p$ is not a weakly visible point. 
		\item $G_p\neq \{p\}$.
		\item For any neighbourhood $U$ of $p$, $G_p\cap U\neq \{p\}$. 
	\end{enumerate}
\end{cor}

\begin{proof} $\:$
	
	(4) $\implies$ (3) Clear.
	
	(3) $\implies$ (2) Let $q\in\partial\O$ be in $G_p$ with $q\neq p$. Then there exists $z_n,w_n\in\O$ tending to $p,q$ respectively with $(z_n|w_n)^\O_o$ tending to infinity. Let $\gamma_n:[0,1]\rightarrow\O$ be a sequence of $\epsilon$-geodesics joining $z_n$ to $w_n$. Suppose that there exists $K\subset\subset\O$ with $K\cap\gamma_n(I_n)\neq\emptyset$ for each $n$. Choose $x_n\in K\cap\gamma_n(I_n)$. Then 
	$$k_\O(z_n,o)+k_\O(w_n,o)  - 2c -\epsilon \leq k_\O(z_n,x_n)+k_\O(w_n,x_n) -\epsilon \leq \lk_\O(\gamma_n)-\epsilon \leq k_\O(z_n,w_n), $$ where $c:=\max_{z\in K} k_\O(z,o)$. In particular, this contradicts with the fact that $(z_n|w_n)^\O_o\rightarrow\infty$. So, we conclude that there doesn't exist such $K\subset\subset\O$. In particular, $\{p,q\}$ is not a pair with visible $\epsilon$-geodesics so $p$ is not a weakly visible point.
	
	(3) $\implies$ (1) Similar to (3) $\implies$ (2).
	
	(2) $\implies$ (4) The assertion follows from Corollary \ref{cor:limitsetandgromovproduct}.
	
	(1) $\implies$ (4) The assertion follows from Remark \ref{rem:remarkabouttheargumentsworkforgeodesics}.
\end{proof}

We now recover \cite[Theorem 15]{NOT} for weak visibility in the case of complete hyperbolicity. Precisely we will prove the following.

\begin{prop}\label{prop:anotherproof}
	Let $\O$ be a complete hyperbolic domain and suppose that $\O$ is hyperbolic at $p\in\partial\Omega$. The following are equivalent. \begin{enumerate}
		\item $p\in\partial\O$ is a weakly visible point.
		\item There exists a neighbourhood $U$ of $p$ such that $\O\cap U$ is complete hyperbolic and $p$ is a weakly visible point for $\O\cap U$.
		\item For any bounded neighbourhood $U$ of $p$ such that $\O\cap U$ is complete hyperbolic, $p$ is a visible point for $\O\cap U$.
	\end{enumerate}
\end{prop}

For the technical property of hyperbolicity at a boundary point $p$ we refer the reader to \cite[Definition 1]{NOT}. Note that bounded domains are hyperbolic at any boundary point. 

The key property here is the additive localization of the Kobayashi distance near visible points given in \cite[Theorem 1.3]{S}. Explicitly, we will make use of \cite[Theorem 9]{NOT} which is a local version of this result.

\begin{lem}\cite[Theorem 1.3]{S}\label{lem:sarkar}
	Let $\O$ be as in one of the cases of Proposition \ref{prop:anotherproof}. Then, for any bounded neighbourhood $U$ of $p$, there exists another neighbourhood $V\subset\subset U$ of $p$ and $C>0$ such that 
	$$	k_{\O\cap U}(z,w)\leq k_\O(z,w) + C \:\:\:\:\: \text{when} \: z,w\in \O\cap V, \: \text{and they are in the same component of} \: \: \O\cap U.
	$$
\end{lem}

\begin{proof}[Proof of Proposition \ref{prop:anotherproof}] By definition, we have that {(3)$\implies$(2)}. The proofs of {(2)$\implies$(1)} and {(1)$\implies$(3)} are similar.
	
	Let $U$ be a bounded neighbourhood of $p$ and choose $V\subset\subset U$ containing $p$ such that $U,V$ satisfy the localization result given in Lemma \ref{lem:sarkar}. Observe that if $z,w,o\in\O\cap V$ then Lemma \ref{lem:sarkar} implies that $$
	(z|w)^{\O\cap U}_o - C \leq (z|w)^{\O}_o \leq (z|w)^{\O\cap U}_o +\frac{C}{2}. $$
	
	Let $D$ be one of $\O,\O\cap U$ and $G$ be the other one. Suppose that $p$ is not a weakly visible point for $D$. Take $o\in\O\cap V$ and observe that Corollary \ref{cor:limitsetandgromovproduct} implies that there exist $q\in \partial\O\cap V$ with $q\neq p$ and sequences $z_n,w_n\in\O\cap V$ tending to $p,q$ respectively with $(z_n|w_n)_o^D$ tending to infinity. Then the inequality above implies that $(z_n|w_n)_o^G$ also tends to infinity hence by Corollary \ref{cor:characterizationofvisibility} we see that $p$ is not a weakly visible point for $G$. We have {(2)$\implies$(1)} and {(1)$\implies$(3)}. \end{proof}

\begin{rem}
	Following the arguments of the proof of \cite[Proposition 3.3]{CMS}, one may observe that we can substitute completeness of $\O$ and $\O\cap U$ with local completeness near $p$. Then, by the monotonicity of the Kobayashi distance we can remove the completeness assumption in the items of Proposition \ref{prop:anotherproof} by demanding that $\O$ is locally complete near $p$, that is there exists a neighbourhood $U$ of $p$ such that
	$\displaystyle{\liminf_{z \rightarrow \partial\O\cap U} k_\O(z,o) = \infty}$ for some (hence any) $o\in\O$. 
\end{rem}

\section{Necessary conditions for non-visibility on convex domains}

Recall that a domain $\O\subset\Cn$ is $\C$-convex if the intersection of $\O$ with any affine complex line is connected and simply connected.

We present the following estimates of Kobayashi distance and the Kobayashi-Royden pseudometric which will be used in the upcoming sections. Note that \eqref{eqn:classicallowerbound} below follows from the proof of \cite[Theorem 5.4]{Blo} in the convex case. 	

\begin{lem}\label{lem:classicallowerbound}\cite[Proposition 2]{NT}
	Let $\O$ be a proper convex (resp. $\C$-convex) domain then
	\begin{equation}\label{eqn:classicallowerbound}
		k_\O(z,w)\geq\left|\dfrac{1}{2}\log\left(\dfrac{\delta_\Omega(z)}{\delta_\Omega(w)}\right)\right|\:\:\:\left(\text{resp.}\:\:\:k_\O(z,w)\geq \left|\dfrac{1}{4}\log\left(\dfrac{\delta_\Omega(z)}{\delta_\Omega(w)}\right)\right|\right)\text{,} \:\:\:\:\: z,w\in\O
	\end{equation}
\end{lem}

\begin{lem}\label{lem:infinitesimallemma}\cite[Proposition 1]{NPZ}
	Let $\O$ be a domain in $\Cn$ containing no affine complex lines then $$\kappa_\O(z;v) \leq \delta^{-1}_\O(z;v).$$
	If, in addition, $\O$ is a convex (resp. $\C$-convex) domain then \begin{equation}\label{eqn:Cconvexinfinitesimal} \k_\O(z;v) \geq \frac{1}{2} \delta^{-1}_\Omega(z;v) \:\:\:\:\: \left( \text{resp. $\k_\O(z;v) \geq \frac{1}{4} \delta^{-1}_\Omega(z;v)$}\right) , \:\:\:\:\: z\in\O, v\in \mathbb{C}^n.\end{equation}
\end{lem}

Recall the notion of local Goldilocks points introduced in \cite[Definition 1.3]{BZ2} following the notion of Goldilocks domains given in \cite[Definition 1.1]{BZ}.

\begin{defn}\label{defn:goldilockspoints}
	Let $\O$ be a domain in $\Cn$ and $p\in\partial\O$. $p$ is said to be a \textit{local Goldilocks point} if there exists a neighbourhood $U$ of $p$ such that
	\begin{enumerate}
		\item $\displaystyle{\int_{0}^{\epsilon}} \frac{M_{\O,U}(r)dr}{r} < \infty$ for some (hence any) $\epsilon>0$, where $$M_{\O,U}(r):=\sup\{ \kappa^{-1}_\O(z;v): z\in\O\cap U \:\:\:\text{with}\:\:\: \delta_\O(z)\leq r, v\in\Cn  \:\:\:\text{with}\:\:\: \|v\|=1 \}.$$
		\item For some (hence any) $z_0\in\O$, there exists $\alpha,C>0$ with $k_\O(z,z_0)\leq C + \alpha|\log(\delta_\O(z))|$.
	\end{enumerate}
\end{defn} 

Note that as the boundaries of convex domains are Lipschitz they satisfy an interior cone condition (see \cite[Definition 2.2]{BZ}). Hence by \cite[Lemma 2.2]{BZ} the second item in the Definition \ref{defn:goldilockspoints} is satisfied for convex domains.

Let $\partial_{lg}\O$ denote the set of local Goldilocks points in $\partial\O$. We say that $p\in\partial\O$ is a \textit{non-Goldilocks point} if $p\notin \partial_{lg}\O$. Recall the following result.

\begin{res}\cite[Theorem 1.4]{BZ2}
	Let $\O$ be a hyperbolic domain such that $\partial\O\setminus\partial_{lg}\O$ is totally disconnected. Then, $\O$ satisfies the visibility property.  
\end{res}

\begin{rem}\label{rem:remarkaboutlimitsetsarenongoldilocks}
	It follows from the proof of the result above that local Goldilocks points cannot be in the limit set of $(\lambda,\epsilon)$-geodesics that tend to the boundary for any $\lambda\geq 1$ and $\epsilon\geq 0$. In particular, by Proposition \ref{prop:limitsetiscontinium} we observe that a domain which fails the visibility property must have a boundary that contains a continuum of non-Goldilocks points.
\end{rem}

The goal of this section is to give a detailed version of Remark \ref{rem:remarkaboutlimitsetsarenongoldilocks} in the case of convexity. 

Let $\Omega$ be a domain in $\mathbb{C}^n$ with $\mathcal C^1$-smooth boundary and $T_p^\mathbb{C}\partial \Omega$ denote the \emph{affine} complex tangent space to $\partial \Omega$ at $p\in\partial \Omega$. Explicitly, $T^\mathbb{C}_p\partial\Omega=p+(F^\mathbb{R}_p\partial\Omega \cap i F^\mathbb{R}_p\partial\Omega)$, where $F^\mathbb{R}_p\partial\Omega$ is the real tangent space to $\partial \Omega$ at $p\in\partial \Omega$. With this terminology, the following result is going to be essential. 

\begin{lem}\cite[Theorem 4.1]{Z}\label{lem:zimmer}
	Let $\O\subset\Cn$ be a convex domain with $\mathcal{C}^{1,\alpha}$-boundary. Suppose that there exists $z_n\rightarrow p\in \partial \O$, $w_n\rightarrow q\in\partial\O$ with $(z_n|w_n)^\O_o\rightarrow \infty.$
	Then $T^\mathbb{C}_p\partial\O=T^\mathbb{C}_q\partial\O.$
\end{lem}


Let $\O$ be a convex domain, $p\in\partial \O$. A face of $\O$ is defined to be $\partial\O\cap H$ where $H$ is a complex hyperplane not intersecting $\O$. \textit{The multiface of $\O$ at $p$}, denoted by $F_p$, is defined as the union of all faces of $\O$ containing $p$. The latter condition in Lemma \ref{lem:zimmer} can be read as follows. 

\begin{rem}
	\label{rem:geometriccond}\cite[Remark 1.5]{Z}
	Let $\O\subset\Cn$ be a convex domain with $\mathcal{C}^1$-boundary, where  $n\geq 2$. Then the following are equivalent. 
	\begin{enumerate}
		\item $T^\mathbb{C}_p \partial \O \neq T^\mathbb{C}_q \partial \O $.
		\item The complex line containing $p,q$ intersects $\O$.
		\item $p$ and $q$ lie on different multifaces of $\O$. 
	\end{enumerate}
\end{rem}

Let us present the main result of this section, which shows that under some boundary regularity the limit set of $\epsilon$-geodesics or Kobayashi geodesics (see Remark \ref{rem:remarkaboutgeometricconditionsextendtogeodesicnonvisiblecase}) tending to the boundary lies on the same complex multiface of convex domains. 

\begin{thm}\label{thm:convexface}
	Let $\O$ be a convex domain in $\Cn$ containing no affine complex lines with $\mathcal C ^{1,\alpha}$-boundary, where $n\geq 2$. Suppose that $\O$ fails the weak visibility property and let $\gamma_n:[0,1]\to\O$ be a sequence of $\epsilon$-geodesics such that $\max_{t\in[0,1]} \delta_\O(\gamma_n(t))\to 0$, $\gamma_n(0)\to p\in \partial \O$ and  $\gamma_n(1)\to q\in \partial\O$ with $q\neq p$. Set $L$ to be as in \eqref{eqn:limitset}. Then $L\subset T^\mathbb{C}_p \partial \O =  T^\mathbb{C}_q \partial \O.$
\end{thm}

\begin{proof}
	
	As $\O$ is complete hyperbolic, by Proposition \ref{prop:commongromovpartner} there exists $r\in \partial \O\setminus\{p,q\}$ with $r\in G_p\cap G_q$. Then by Lemma \ref{lem:zimmer} we have that $
	T^\mathbb{C}_p \partial \O =  T^\mathbb{C}_r \partial \O =  T^\mathbb{C}_q \partial \O$. It again follows from Proposition \ref{prop:commongromovpartner} that any $s\in L$ is in $G_p\cup G_q$, hence Lemma \ref{lem:zimmer} shows that $ T_s^\mathbb{C} \partial \O = T_p^\mathbb{C} \partial \O = T_q^\mathbb{C} \partial \O $. 
\end{proof}

Unfortunately, if $\lambda>1$ we do not know whether the same assertion holds for limit sets of $(\lambda,\epsilon)$-geodesics tending to the boundary.

For domains in $\mathbb{C}^2$, Theorem \ref{thm:convexface} provides even more detail.

\begin{cor}\label{cor:nonvisibilityinc2}
	Let $\O$ be a convex domain in $\C^2$ containing no affine complex lines with $\mathcal{C}^{1,\alpha}$-boundary. Suppose that $\O$ fails the weak visibility property. Then there exists a complex tangential line segment $S\subset\partial\O$ of non-Goldilocks points.
\end{cor}

\begin{proof}
	Suppose that $\O$ is not weakly visible, let $\gamma_n:[0,1]\to\O$  be a sequence of $\epsilon$-geodesics with $\max_{t\in[0,1]}\delta_\O(\gamma_n(t))\to 0$, $z_n:=\gamma_n(0)\to p$, $w_n:=\gamma_n(1)\to q$ with $q\neq p$. Let $L$ be as in \eqref{eqn:limitset}. Then it follows by Theorem \ref{thm:convexface} that $L\subset F_p$. By Remark \ref{rem:remarkaboutlimitsetsarenongoldilocks} $L$ contains only non-Goldilocks points. If $L$ contains a line segment we are done. 
	
	Suppose that $L$ does not contain a line segment. Then, we have another point $r$ in $L$ such that $p,q$ and $r$ are not collinear. Consider the convex hull of $p,q,r$ denoted by $C$. As $\O$ is convex, by Remark \ref{rem:geometriccond} we have that $C\subset\ov\O\cap T^\mathbb{C}_p\partial\O$. This shows that $\partial \O\cap T_p^\mathbb{C}\partial \O$ contains a non-empty open set in the topology of $T_p^\mathbb{C}\partial \O$. It is then easy to see that Lemma \ref{lem:infinitesimallemma} implies that interior points of $C$ (in the topology of $T_p^\mathbb{C}\partial \O$) consist of non-Goldilocks points. \end{proof}

\begin{rem}\label{rem:remarkaboutgeometricconditionsextendtogeodesicnonvisiblecase}
	Theorem \ref{thm:convexface} and Corollary \ref{cor:nonvisibilityinc2} extend to the case where $\O$ fails the geodesic visibility property.
\end{rem}

Unfortunately, the proof of Corollary \ref{cor:nonvisibilityinc2} does not extend to convex domains in $\Cn$ where $n\geq 3$. Namely in this case, if the limit set of $\epsilon$-geodesics $L$ does not contain a line segment, the plane containing $L$ may not be a complex line. In this case, we have no argument to see the existence of a line segment of non-Goldilocks points in the boundary. 

\section{Sufficient conditions for non-visibility in ($\C$-)convex domains}\label{sec:sufficientconditions}

\subsection{Motivating example: Rate of contact with the complex tangent space}

Let $\O\subset\Cn$ be a ($\mathbb C$-)convex domain, $p\in\partial\O$ and $U$ be a neighbourhood of $p$. By \eqref{eqn:Cconvexinfinitesimal} one may see that the first condition in Definition \ref{defn:goldilockspoints} can be characterized in terms of the rate of contact with the complex tangent space near $p$. Explicitly, $p\in\partial\O$ is a non-Goldilocks point if and only if for any sufficiently small neighbourhood of $U$ of $p$ we have some $\epsilon>0$ with
\begin{equation}\label{eqn:goldilocksconvexnew}\int^{\epsilon}_{0} \dfrac{N_{\O,U}(r)}{r} dr = \infty,\end{equation} where $ N_{\O,U}:= \sup\{\delta_\O(z;v):z\in\O\cap U \:\text{with}\:\delta_\O(z)\leq r, v\in\Cn\:\text{with}\:\|v\|=1\}$.

Consider the special type of domains defined by 
\begin{equation}
	\O_\Psi:=\{(z_1,z_2)\in\C^2:\Re z_2\geq \Psi(\Re z_1 )\},\label{eqn:modeldomains}
\end{equation}
where $\Psi:\mathbb R \rightarrow [0,\infty)$ is a smooth even convex function, increasing on $[0,\infty)$, verifying 
$\Psi(0)=0$. Since $\O_\Psi$ is a convex domain, the second item of Definition \ref{defn:goldilockspoints} is satisfied on $\O_\Psi$. Moreover, by \eqref{eqn:goldilocksconvexnew} it follows from \cite[Theorem 1.4]{BZ2} and a direct calculation that if $\Psi$ verifies \begin{equation}
	\label{eqn:examplesofvisible}\int_{0}^{\varepsilon} \frac{\Psi^{-1}(x)}{x}dx < \infty\end{equation} for some (hence any) $\varepsilon>0$ then $\O_\Psi$ satisfies the visibility property.  

On the other hand, note that if $\Psi$ satisfies \begin{equation}\label{eqn:profthomasresult}
	\Psi(x) = o( e^{\frac{-{\pi}}{2x}}), \:\:\:\:\: x\:\: \text{near} \:\:0,
\end{equation} then it follows by \cite[Proposition 5.4]{BNT} that $\O_\Psi$ fails the geodesic visibility (hence the weak visibility) property. In fact, we can easily extend this.
\begin{prop}\label{prop:exampleofnonvisiblemodeldomain}
	Let $\O:=\O_\Psi$ be defined as in \eqref{eqn:modeldomains}	and suppose that there exists $c>0$ with \begin{equation}\label{eqn:exampleofnonvisible}
		\Psi(x) = o(e^{-\frac{c}{x}}), \:\: x\:\: \text{near} \:\:0.
	\end{equation} Then $\O_\Psi$ fails the geodesic visibility (hence the weak visibility) property.
\end{prop}

\begin{proof}
	Let $\O$ be as above. Note that we can choose a constant $r>0$ such that the biholomorphism (of $\mathbb C ^2$) $f_r:\O\rightarrow\mathbb C^2$ defined by $(z_1,z_2)\mapsto(r z_1, z_2)$ maps $\O$ onto $\O'$ where a defining function of $\O'$ verifies \eqref{eqn:profthomasresult}. In particular, by Proposition \ref{prop:commongromovpartner} it follows that there exists $p'\in \partial\O'$ with $G_{p'}\neq \{p'\}$. Clearly, the map $f_r$ preserves Euclidean distance up to a multiplicative constant. Then by the biholomorphic invariance of the Kobayashi distance, we conclude that there exists $p\in\partial\O$ with $G_p\neq \{p\}$. By Corollary \ref{cor:characterizationofvisibility} we are done.
\end{proof}

In the examples above, seemingly, the Euclidean distance between points of $\partial\O\cap\{\text{Re} z_1=0\}$ plays a part in whether or not they are a weakly visible or geodesic visible pair.  See the proof of \cite[Proposition 5.4]{BNT} for details.

Proposition \ref{prop:exampleofnonvisiblemodeldomain} leaves us with a gap, as there are domains $\O_\Psi$ defined as in \eqref{eqn:modeldomains} where the function $\Psi$ does not satisfy the assumptions given in either of \eqref{eqn:examplesofvisible} and \eqref{eqn:exampleofnonvisible}. We will show that for such domains we can find non-visible $(\lambda,0)$-geodesics for any $\lambda > 1$.

Let $\O$ be a domain in $\C^2$ containing no affine complex lines and $\pi_\O:\O\to\partial\O$ denote the projection to the boundary. If the closest point to $z$ in $\partial\O$ is not unique, define $\pi_\O(z)$ to be one of the points $z'\in\partial\O$ verifying $\|z'-z\|=\delta_\O(z)$. Let $\{X_z,Y_z\}$ be a basis consisting of complex tangential and complex normal unit vectors to $\partial\O$ at $\pi_\O(z)$ respectively. Explicitly, let $X_z\in\C^2$ be a unit vector in $T^\C_{\pi_\O(z)}\partial\O$ and $Y_z:= (\pi_\O(z)-z)/\|\pi_\O(z)-z\|$. For $v\in\C^2$ set $v^\perp_z:= \langle v, X_z \rangle_\C X_z$ and $v_z:=\langle v, Y_z \rangle_\C Y_z$. With this notation, we have: 
\begin{lem}\cite[Proposition 14]{NPZ}\label{lem:NPZlemma}
	Let $\O$ be a convex domain in $\C^2$ containing no affine complex lines. There exists $c>0$ such that 
	\begin{equation}\label{eqn:NPZtriangleinequality}
		\k_\O(z;v_z) + \k_\O(z;v^\perp_z) \geq \k_\O(z;v) \geq c (\k_\O(z;v_z) + \k_\O(z;v^\perp_z)), \:\:\: \: \: z\in \O, v\in \mathbb{C}^2,
	\end{equation}
\end{lem} 

The lemma above is a weaker form of \cite[Proposition 14]{NPZ}, which remains true on $\C$-convex domains in $\Cn$ in terms of the minimal basis (see for instance \cite[p. 3223]{NPT}). Note that the upper bound above follows as a consequence of convexity. Lempert's theorem \cite{L} implies that if $\O$ is convex the Kobayashi-Royden pseudometric of $\O$ coincides with its Carath\'eodory-Reiffen pseudometric (see for instance \cite[p. 22]{JP}). Then the upper bound in \eqref{eqn:NPZtriangleinequality} easily follows. 


Lemma \ref{lem:NPZlemma} leads to the following proposition.

\begin{prop}\label{prop:example}
	Let $\O:=\O_\Psi$ be given as in \eqref{eqn:modeldomains} and suppose that $\Psi$ does not verify \eqref{eqn:examplesofvisible}.Then any $p\in\partial\O\cap\{\Re z_1=0\}$ is not a visible point. Explicitly, for any $\lambda>1$, there exist $(\lambda,0)$-geodesics $\gamma_n:[0,1]\to\O$ satisfying $\max_{t\in[0,1]}\delta_\O(\gamma_n(t))\to 0$, $\gamma_n(0)\to p$ and $\gamma_n(1)\not\to p$.
\end{prop}

\begin{proof} 
	We may assume that $p=(0,0)$. 
	
	Set $p_n:=(0,1/n)$ and $\gamma_n:[0,1]\rightarrow \O$ to be the curve given by $\gamma_n(t):=(it,f_n(t))$ where $f_n$ is an increasing function, $\gamma_n(0)=p_n$ and $\delta_\O(\gamma_n(t);(i,0))=c \delta_\O(\gamma_n(t);(0,f_n'(t)))$. Note that such a curve exists by the existence and uniqueness theorem of ODE's.
	
	\textit{Claim 1:} $$ \displaystyle{\lim_{n\to \infty} \max_{t\in[0,1]}\delta_\O(\gamma_n(t)) = 0}. $$
	
	\textit{Proof of Claim 1:} 
	A direct calculation shows that \begin{equation}\label{eqn:exampledirectionaldistance}
		\delta^{-1}_\O(\gamma_n(t);(i,0)) = 1/\Psi^{-1}(f_n(t)) \:\:\:\:\: \text{and} \:\:\:\:\: \delta^{-1}_\O(\gamma_n(t);(0,f_n'(t)))=f'_n(t)/f_n(t).\end{equation}
	
	Therefore our assumption gives \begin{equation}\label{eqn:thehit}
		\dfrac{\Psi^{-1}(f_n(t))f'_n(t)}{f_n(t)} = c^{-1}.
	\end{equation}
	
	Let $  D_n:= \max_{t\in[0,1]}\delta_\Omega(\gamma_n(t)) = \delta_\Omega(\gamma_n(1)) .$ Then, by \eqref{eqn:thehit} we have 
	$$ c^{-1}=\int_{0}^{1} c^{-1} dt = \int_{0}^{1}\dfrac{\Psi^{-1}(f_n(t))f'_n(t)}{f_n(t)}dt = \int_{1/n}^{D_n} \dfrac{\Psi^{-1}(y)dy}{y} .$$
	
	Suppose that $D_n\not\to0$. Then, the inequality above contradicts the fact that the integral in \eqref{eqn:examplesofvisible} diverges. Claim 1 follows. 
	
	\textit{Claim 2:} For large $n$, $\gamma_n$ is a $(1+2c,0)$-geodesic.
	
	\textit{Proof of Claim 2.} By Claim 1, the images of $\gamma_n$ tend to the boundary. Then for large $n$ and $t\in[0,1]$ a basis consisting of a complex tangential and a complex normal unit vector to $\partial\O$ at $\pi_\O(\gamma_n(t))$ is given by  $\{(1,0),(0,1)\}$ respectively. It follows by Lemma \ref{lem:infinitesimallemma} and \eqref{eqn:exampledirectionaldistance} that $\k_\O(\gamma_n(t);(0,f'_n(t)))\geq \delta^{-1}_\O(\gamma_n(t);(0,f'_n(t))) /2 = {f'_n(t)}/{2f_n(t)}$. Furthermore, by considering the embedding of the right half plane to $\O$ given by $z\mapsto(0,z)$ we get $\k_\O(\gamma_n(t);(0,f'_n(t)))\leq \delta^{-1}_\O(\gamma_n(t);(0,f'_n(t))) /2 = {f'_n(t)}/{2f_n(t)}$ hence $\k_\O(\gamma_n(t);(0,f'_n(t)))={f'_n(t)}/{2f_n(t)}$. Then  Lemma \ref{lem:infinitesimallemma}, Lemma \ref{lem:NPZlemma} and \eqref{eqn:exampledirectionaldistance} show that for large $n$ and $t\in[0,1]$ we have 
	\begin{equation}\label{eqn:exampletriangle}
		\k_\O(\gamma_n(t);\gamma_n'(t)) =  \k_\O(\gamma_n(t);(i,f_n'(t))) \leq  \k_\O(\gamma_n(t);(0,f_n'(t))) + \k_\O(\gamma_n(t);(i,0)) \leq \end{equation} $$ (1+2c) \k_\O(\gamma_n(t);(0,f_n'(t))) = \dfrac{(1+2c) f'_n(t)}{2f_n(t)} .$$
	Let $t_1<t_2\in[0,1]$ be arbitrary, then by \eqref{eqn:exampletriangle} we have 
	$$ \lk_\O(\gamma_n|_{[t_1,t_2]}) \leq (1+2c)\displaystyle{\int_{t_1}^{t_2}\dfrac{f_n'(t)}{2f_n(t)}dt = \dfrac{(1+2c)}{2}\log\dfrac{f_n(t_2)}{f_n(t_1)}=\dfrac{(1+2c)}{2}\log\dfrac{\delta_\O(\gamma_n(t_2))}{\delta_\O(\gamma_n(t_1))}} .$$
	
	On the other hand by Lemma \ref{lem:classicallowerbound} we have $$ k_\O(\gamma_n(t_1),\gamma_n(t_2)) \geq \dfrac{1}{2}\log\dfrac{\delta_\O(\gamma_n(t_2))}{\delta_\O(\gamma_n(t_1))} .$$  Claim 2 follows. 
	
	By choosing $c >0$ as small as we like, by Claims 1 and 2 we are done.   
\end{proof}

Unfortunately, we were not able to decide whether the domains $\O_\Psi$ of the form \eqref{eqn:modeldomains} where $\Psi$ fails both \eqref{eqn:examplesofvisible} and \eqref{eqn:exampleofnonvisible} satisfy the geodesic visibility (hence the weak visibility) property or not.

\subsection{A generalization: Strongly non-Goldilocks points and non-visibility}

The goal of this section is to extend Proposition \ref{prop:example} to the general case for convex domains with $\mathcal{C}^1$-boundary. 

\begin{defn}\label{defn:stronglynongoldilocksdefinition}
	Let $\O$ be a domain in $\Cn$. A $\mathcal C^{1}$-smooth boundary point $p\in\partial\O$ is said to be a \emph{weakly Goldilocks point} if for some $\epsilon>0$ we have 
	$$
	\displaystyle{\int_0^\epsilon \dfrac{\sup\{\k^{-1}_\O(p_{r'};v):v\in\Cn, \:\:\: \|v\|=1, \:\:\:r'\leq r\}}{r}dr < \infty},
	$$
	where $p_r:=p+r\eta_p$ and $\eta_p$ denotes inner unit normal to $\partial\O$ given at $p$. 
\end{defn}
Points which are not weakly Goldilocks are called \emph{strongly non-Goldilocks}.
It is clear that being a locally Goldilocks point in the sense of Definition \ref{defn:goldilockspoints} implies being a weakly Goldilocks point. A counter-example in the next subsection shows that there are points which are not locally
Goldilocks, but still are weakly Goldilocks,

If $\O\subset\Cn$ is a $\C$-convex domain then \eqref{eqn:Cconvexinfinitesimal} implies that $p\in\partial\O$ is a strongly non-Goldilocks point if and only if \begin{equation}\label{eqn:stronglynonGoldiconvexcase}
	\displaystyle{\int_0^\epsilon \dfrac{\sup\{\delta_\O(p_{r'};v):v\in\Cn, \:\:\: \|v\|=1, \:\:\:r'\leq r\}}{r}dr=\infty}.
\end{equation} 

In fact, we can say even more if $\O\subset\C^2$ is convex.

\begin{prop}\label{prop:c2nongoldilocksissimpler}
	Let $\O$ be a convex domain in $\C^2$, $p\in\partial\O$ be a $\mathcal C^1$-smooth boundary point and $X\in T^\C_p \partial \O$ be a unit vector. Set $M_{\O,p}(r):=\delta_\O(p_r;X).$ Then $p$ is a strongly non-Goldilocks point if and only if 
	$$\displaystyle{\int_0^\epsilon \dfrac{M_{\O,p}(r)}{r}dr=\infty}.$$ 
\end{prop}

\begin{proof}
	As the other implication is clear, we only need to show that if $p$ is a strongly non-Goldilocks point then we must have $\int_0^\epsilon {M_{\O,p}(r)}/{r}dr=\infty$.
	Assume that $p\in\partial\O$ is a strongly non-Goldilocks point. 
	
	\textit{Claim 1:} There exists $C>0$ depending on $\O$, and $\delta_1>0$ depending on $p$ with $$\sup\{\delta_\O(p_{r};v):v\in\C^2, \:\:\: \|v\|=1\} \leq C M_{\O,p}(r)$$ if $r < \delta_1$. 
	
	To see Claim 1, let $\delta>0$ be small enough so that for $r\in(0,\delta)$ the closest point to $p_r$ in $\partial\O$ is given by $p$. Let $Y:=(p_r-p)/\|p_r-p\|$ and $v\in\C^2$ with $\|v\|=1$ be arbitrary. Write $v=\alpha X + \beta Y$ where $\alpha,\beta\in \C$ with $|\alpha|^2+|\beta|^2=1$.  Then by Lemma \ref{lem:infinitesimallemma} and Lemma \ref{lem:NPZlemma} we have  $$\delta_\O(p_r;v)\leq \k^{-1}_\O(p_r;v) \leq c^{-1} (|\alpha| \k_\O(p_r;X)+|\beta| \k_\O(p_r;Y))^{-1} \leq \dfrac{1}{2c} \k^{-1}_\O(p_r;X) \leq \dfrac{1}{c} M_{\O,p}(r).$$ Note that above we use the fact that $\k_\O(p_r;X) \leq \k_\O(p_r;Y)$ which follows from boundary regularity. Setting $\delta_1:=\delta$ and taking the supremum over $v$, we have Claim 1. 
	
	\textit{Claim 2:} There exists $\delta_2>0$ such that $M_{\O,p}(r)$ is a monotone increasing function on $(0,\delta_2)$.
	
	Note that $M_{\O,p}(r)$ is a concave function of $r$. To see this choose $r_1 < r_2$ small enough and consider the analytic discs given by $D_1(\zeta):=p_{r_1}+\zeta\delta_\O(p_{r_1};X) X$, $D_2(\zeta):=p_{r_2}+\zeta\delta_\O(p_{r_2};X)X$. Take $t\in[0,1]$ and consider the point $p^t:=t p_{r_1}+(1-t)p_{r_2} \in \O$. As $D_1(\Delta),D_2(\Delta) \subset \O$ we have by convexity that the image of $D_t(\zeta):= t D_1(\zeta) + (1-t) D_2(\zeta)$ lies in $\O$. This implies that $M_{\O,p}(t r_1 + (1-t)r_2)= \delta_\O(p^t;X) \geq t\delta_\O(p_{r_1};X)+(1-t)\delta_\O(p_{r_2};X) = t M_{\O,p}(r_1) + (1-t)M_{\O,p}(r_2)$. Now, since $M_{\O,p}(r)$ is a concave function of $r$, it attains its maximum value at (at most) one point, say $\delta_2$, thus $M_{\O,p}(r)$ is a monotone increasing function on $(0,\delta_2)$ hence Claim 2 follows. 
	
	Now, consider $r < \delta$ where $\delta:=\min\{\delta_1,\delta_2\}$ with $\delta_1$ as in Claim 1 and $\delta_2$ as in Claim 2. Then there exists $C>0$ depending on $\O$ such that for $r\in (0,\delta)$ we have $$ \sup\{\delta_\O(p_{r'};v):v\in\C^2, \:\:\: \|v\|=1, \:\:\:r'\leq r\}\leq C M_{\O,p}(r).$$ By Claims 1 and 2 and \eqref{eqn:stronglynonGoldiconvexcase} we are done.
\end{proof}

Actually, if $\O\subset\Cn$ (not necessarily $n=2$) is a convex domain containing no affine complex lines with no non-trivial analytic discs in the boundary, then the proof of \cite[Lemma 2.3]{B} (see the result of D'Addezio in \cite[Lemma A.1]{BG}) implies that $$\lim_{r\to 0}\sup\{\delta_\O(z;v): \delta_\O(z)\leq r, \:\:\: \|v\|=1\}=0.$$ Therefore if $\O$ has $\mathcal C^{1,1}$-smooth boundary a stronger statement than Claim 2 above holds. Namely, in that case the uniform inner ball condition implies that there exists $\delta>0$ such that $$\sup\{\delta_\O(z;v):v\in\Cn, \:\:\: \|v\|=1, \:\:\: \delta_\O(z)\leq r \}=\sup\{\delta_\O(z;v):v\in\Cn, \:\:\: \|v\|=1, \:\:\: \delta_\O(z)= r\}$$ for $r\in(0,\delta).$

One generalization of Proposition \ref{prop:example} is the following:

\begin{thm}\label{thm:nonvisibleinC2}
	Let $\O\subset\C^2$ be a convex domain containing no affine complex lines with $\mathcal C^1$-smooth boundary and $p\in\partial \O$. Suppose that $p$ is a strongly non-Goldilocks point and $F_p\neq\{p\}$. Then $p$ is not a visible point. Explicitly, there exists $\lambda>1$ such that $\O$ has $(\lambda,0)$-geodesics $\gamma_n:[0,1]\to\O$ with  $\max_{t\in[0,1]}\delta_\O(\gamma_n(t))\to 0$, $\gamma_n(0)\to p$ and $\gamma_n(1)\not\to p$.
\end{thm}

\begin{proof}
	Let $q\in F_p$ such that $q\neq p$ then $X=(p-q)/\|(p-q)\|\in T_p^\C \partial \O$. Note that boundary regularity and convexity show that the inner unit normal to $\partial \O$ coincides at any point in $F_p$.	
	
	\textit{Claim 1:} There exists a line segment $S\subset F_p$ containing $p$ and a function $H:(0,\delta)\to\mathbb{R}^+$ verifying 
	$ \int_0^\delta H(r)/r dr = \infty$ and $M_{\O, s}(r) \geq H(r) $ for any $s\in S$ and $r\in(0,\delta). $
	
	To see Claim 1, note that by the proof of Proposition \ref{prop:c2nongoldilocksissimpler} there exists $\delta>0$ such that $M_{\O,p}(r)$ is monotone increasing function on $r\in(0,\delta)$ and $\int_{0}^\delta M_{\O,p}(r)/r dr = \infty.$	Let $S=\{ (1-t) p + t q: t \in [0,1/2]\}$, $f_r(\zeta):=p_r+ \zeta \delta_\O(p_r;X) X$ and $f_{r,t}(\zeta):=(1-t) f_r(\zeta)  + t q_r$. Note that by convexity $S\subset F_p$ and the images of $f_{r,t}$ lie in $\O$. Consequently $ \delta_\O(((1-t) p+ t q)_r;X) \geq 1/2 M_{\O,p}(r)$. Setting $H(r):= 1/2 M_{\O,p}(r)$, Claim 1 follows. 
	
	Now let $S, H(r)$ be as in Claim 1. We may assume that $p=(0,0), q=(i\tau, 0)$ for some $\tau>0$ so we may write $S=\{(it,0):t\in[0,\tau/2]\}$. For any $n\geq 1$ consider the curves $\gamma_n(t):=(i\tau t/2,f_n(t))$ where $f_n:[0,1]\rightarrow \mathbb{R}$ is a $\mathcal C^1$-smooth increasing function satisfying $f_n(0)=1/n$. Suppose furthermore that $\gamma_n$ verifies $\delta_\O(\gamma_n(t);(0,f'_n(t)))=\delta_\O(\gamma_n(t);(i \tau/2,0)) .$ Such a curve exists by the existence and uniqueness theorem of ODE's. 
	
	\textit{Claim 2:} $$ \displaystyle{\lim_{n\to \infty} \max_{t\in[0,1]} \delta_\O(\gamma_n(t)) = 0} .$$ Thus for large enough $n$ and $t\in [0,1]$ we have $\delta_\O(\gamma_n(t))=f_n(t)$.
	
	\textit{Proof of Claim 2:} Note that $$\delta^{-1}_\O(\gamma_n(t);(0,f'_n(t)))=f'_n(t)/f_n(t) \:\:\: \text{and} \:\:\: \delta^{-1}_\O(\gamma_n(t);(i \tau/2,0)) \leq \tau/H(f_n(t)) .$$ Then the proof follows from the same arguments we provided in the proof of Proposition \ref{prop:example}, Claim 1. The only difference is we replace $\Psi^{-1}(f_n(t))$ with $H(f_n(t))$.
	
	\textit{Claim 3:} Each $\gamma_n$ is a $(4,0)$-geodesic.  
	
	\textit{Proof of Claim 3:} Let $t_1\leq t_2 \in [0,1]$ be arbitrary. Then as $\gamma_n([0,1])$ tend to the boundary and for large $n$ and $t\in [0,1]$ we have $\delta_\O(\gamma_n(t))=f_n(t)$, Lemma \ref{lem:classicallowerbound} and Lemma \ref{lem:NPZlemma} implies that
	
	$$ k_\O(\gamma_n(t_1),\gamma_n(t_2)) \geq \dfrac{1}{2}\log\left(\dfrac{f_n(t_2)}{f_n(t_1)}\right) $$ and
	
	$$ \displaystyle{\lk_\O(\gamma_n |_{[t_1,t_2]}) = \int_{t_1}^{t_2} \k_\O(\gamma_n(t);\gamma'_n(t))dt \leq \int_{t_1}^{t_2} \k_\O(\gamma_n(t);(i \tau / 2,0))dt+\int_{t_1}^{t_2} \k_\O(\gamma_n(t);(0,f'_n(t)))dt \leq } $$ $$  
	\displaystyle{ 2 \int_{t_1}^{t_2} \delta^{-1}_\O(\gamma_n(t);(0,f'_n(t)))dt =  2 \int_{t_1}^{t_2} \dfrac{f'_n(t)}{f_n(t)}dt  = 2 \log\left(\dfrac{f_n(t_2)}{f_n(t_1)}\right) } .$$
	
	Hence we have Claim 3.  
	
	By Claims 2 and 3, taking $\lambda = 4$ we are done. 
\end{proof}

Theorem \ref{thm:nonvisibleinC2} extends to $\C$-convex domains in the following way:

\begin{cor}\label{cor:corollaryforCconvex}
	Let $\O\subset\C^2$ be a $\C$-convex domain containing no affine complex lines, $p\in\partial \O$, $X\in T^\C_p\partial\O$ be a unit vector vector and $S\subset T^\C_p\partial\O\cap \partial\O$ be a line segment containing $p$ such that for each $s\in S$ we have $T^\mathbb{R}_p \partial \O = T^\mathbb{R}_s \partial \O$. Suppose that there is a function $H:[0,\delta]\rightarrow\mathbb{R}$ such that for $s\in S,r\in (0,\delta)$ we have $\delta_\O(s_r;X) \geq H(r)$ and $$\displaystyle{\int_{0}^{\delta} \dfrac{H(r)}{r} dr = \infty}.$$ Then $\O$ fails the visibility property. Explicitly, there exists $\lambda>1$ such that for any $s\in S$,  $\O$ has $(\lambda,0)$-geodesics $\gamma_n:[0,1]\to\O$ with ${\max_{t\in[0,1]}}\delta_\O(\gamma_n(t))\to 0$, $\gamma_n(0)\to p$ and $\gamma_n(1)\to s$. 
\end{cor}
\begin{proof}
	Note that the main use of convexity in the proof of Theorem \ref{thm:nonvisibleinC2} is in the proof of Claim 1. Corollary \ref{cor:corollaryforCconvex} includes the conclusion of Claim 1 in its assumptions. Then, the proof of Corollary \ref{cor:corollaryforCconvex} follows from the arguments provided in the proofs of Claims 2 and 3 in the proof of Theorem \ref{thm:nonvisibleinC2}. The only difference is to use the weaker estimates given in $\C$-convex cases of Lemmas \ref{lem:classicallowerbound} and \ref{lem:infinitesimallemma} and to apply the complete version of \cite[Proposition 14]{NPZ}. 
\end{proof}

We are curious to see if the converse of Theorem \ref{thm:nonvisibleinC2} holds.

\subsection{An application: Gromov non-hyperbolicity of convex domains}\label{sec:new}

In higher dimensions the complex line containing the complex tangential line segment may not agree with the complex line with the sufficient rate of contact, making the situation more complicated. To give an extension of Theorem \ref{thm:nonvisibleinC2} to higher dimensions we give the following definition. 

\begin{defn}
	Let $\Omega$ be a convex domain in $\Cn$ containing no affine complex lines and let $S\subset\partial \Omega$ be a line segment. We say that such a line segment is a \emph{non-Goldilocks line segment} if: \begin{enumerate}
		\item There exists a complex hyperplane $H\subset\Cn$ such that $S \subset H\cap \partial \Omega$.
		
		\item There exists $\delta_0 > 0$ and a unit vector $X \in \mathbb{C}^n$ such that for any $p\in S$ and $r\in (0,\delta_0)$ $p_r:=p+r X \in \Omega$ and $\delta_\Omega(p_r)=\|p-p_r\|=r$.   
		
		\item Let $s,q$ be two distinct points of $S$ and set $Y:=(s-q)/\|s-q\|$. There exists a function $H:(0,\delta_0)\to \mathbb{R}^+\cup\{0\}$ such that $\delta_\O(p_r;Y)\geq H(r)$ for all $p\in S$ and $r\in(0,\delta_0)$, moreover $$\displaystyle \int_{0}^{\delta_0} \dfrac{H(r)}{r} dr = \infty.$$
	\end{enumerate}
\end{defn}

It is easy to see from its proof that the domains given in Theorem \ref{thm:nonvisibleinC2} contain a non-Goldilocks line segment in the boundary. In fact, by applying the complete version of \cite[Proposition 14]{NPZ} and imitating the proof of Theorem \ref{thm:nonvisibleinC2} one may extend it to higher dimensions as follows.

\begin{cor}\label{cor:howtheresultextendtohigherdimensions}
	Let $\O\subset\Cn$ be a convex domain containing no affine complex lines and assume that $\partial \Omega$ contains a non-Goldilocks line segment $S$. Then any $p\in S$ is not a visible point. 
\end{cor}

For brevity, we refer the reader to \cite[Part III]{BH} for details on Gromov hyperbolicity. Recall the following result: 

\begin{res}\cite[Theorem 1.5]{BGZ}
	Let $\Omega$ be a Gromov hyperbolic convex domain containing no affine complex lines. Then the identity map extends to a homeomorphism between Gromov compactification of $\Omega$ (with respect to the Kobayashi distance) and the end compactification of its Euclidean closure. 
\end{res}

In view of the result above, visibility property of Gromov hyperbolic metric spaces with respect to their Gromov boundary (see for instance \cite[p. 428, Lemma 3.2]{BH}) and the geodesic stability lemma (see for instance \cite[p. 401, Theorem 1.7]{BH}) imply the following remark:

\begin{rem}\label{rem:filipposremark}
	Let $\Omega$ be a Gromov hyperbolic convex domain containing no affine complex lines. Then $\Omega$ satisfies the visibility property. 
\end{rem} 

Thanks to Corollary \ref{cor:howtheresultextendtohigherdimensions} and the remark above, we give the following Gromov non-hyperbolicity criterion for convex domains with no boundary regularity. 

\begin{thm}
	Let $\O\subset\Cn$ be a convex domain containing no affine complex lines and assume that $\partial \Omega$ contains a non-Goldilocks line segment $S$. Then $\Omega$ is not Gromov hyperbolic.
\end{thm}

Notably, the above result improves \cite[Theorem 1.6]{Z2}.

\subsection{Two examples: Goldilocks points, weakly Goldilocks points, and visible points}

\begin{rem}\label{rem:convexample} There exists a convex domain $\O$ with smooth boundary, and $p\in\partial\O$ such that $p\notin \partial_{lg}\O$ but $p$ is a weakly Goldilocks point.
\end{rem}

\begin{proof}
	Let $\Psi:\mathbb{R}\to\mathbb{R}$ defined by $\Psi(0):=0$ and $\Psi(x):=e^{-1/\sqrt{|x|}}$ if $x\neq 0$. For $j\in \mathbb{N}$, let $t_j:=2^{-j}$,
	and $L_j(x):=(\Psi(t_j)-\Psi(t_{j+1}))(x-t_{t+j})/(t_{j}-t_{j+1})+ \Psi(t_{j+1})$. Let $\Psi_0:\mathbb{R}\to\mathbb{R}$ be given by $\Psi_{0}(x):=\max\{\Psi(x),L_n(x):n\in 2\mathbb{N}\}$. Since $\Psi$ is convex and $\Psi(x)\leq L_j(x)$ if and only if $x\in [t_{j+1},t_j]$, $\Psi_0(x)$ is a continuous convex function. 
	
	Let $\rho:\mathbb{R}\to\mathbb{R}^+\cup\{0\}$ be a smooth even function with  $\text{supp}(\rho)\subset\subset(-1,1)$ and $\int_{-\infty}^{\infty} \rho(x) dx = 1$. Set $\rho_\epsilon(x):=\rho(x/\epsilon)/\epsilon$, then $\text{supp}(\rho_\epsilon)\subset\subset(-\epsilon,\epsilon)$ and $\int_{-\infty}^{\infty} \rho_\epsilon(x) dx = 1$. 
	
	For $j\in \mathbb{N}$ set $\Psi_\infty(x):=(\Psi_{0}\ast \rho_{t_j/4})(x)$ if $x\in [3t_j/4,5t_j/4]$ and $\Psi_{\infty}(x)=\Psi_0(x)$ otherwise. 
	
	\textit{Claim 1:} $\Psi_\infty$ is a smooth convex function and $\Psi_\infty(x)\geq \Psi_0(x)$.
	
	For $j>0$ and $x \in [3t_j/4,5t_j/4]$ we have that 
	$\Psi_\infty^{(2)}(x) = (\Psi_{0}^{(2)}\ast \rho_{t_j/4})(x)$ hence it follows from the convexity of $\Psi_0$ that $\Psi_\infty$ is also a convex function. Furthermore since $\rho$ is even and $\Psi_0$ is convex, by definition of $\Psi_\infty$ one concludes that $\Psi_\infty(x)\geq\Psi_0(x)$. Note that due to our construction, the intervals $[3t_j/4,5t_j/4]$ are of positive distance with one another. It is not difficult to see that this implies smoothness of $\Psi_\infty$ away from the origin. To see that it is smooth at the origin we will show that \begin{equation}\label{eqn:thisfunctionissmooth}
		\lim_{x\to 0}{\Psi_\infty^{(n)}(x)}=0, \:\:\:\:\: \forall n\in \mathbb{N}.
	\end{equation}
	
	Set $S_0:= \cup_{n=0}^\infty [3t_{n}/4,5t_{n}/4]$ and $S_1:= \mathbb{R}\setminus S_0$. Then $\Psi_\infty(x)=\Psi_0(x)$ for $x\in S_1$, consequently $$\lim_{S_1 \ni x\to 0}{\Psi_\infty^{(n)}(x)}=0, \:\:\:\:\: \forall n\in \mathbb{N}.$$
	
	Fix $j>0$ and let $x \in [3t_j/4,5t_j/4]$. Then, 
	$\Psi_\infty^{(n)}(x) = (\Psi_{0}\ast \rho^{(n)}_{t_j/4})(x)$ so $$ \|\Psi_\infty^{(n)}(x)\| \leq \frac{ 2^{2n+1} C_n \Psi_0(5t_{j}/4) } {t_j^n} \leq \frac{ 2^{2n+1} C_n \Psi(2 t_j) } {t_j^n}, $$ where $C_n:=\max_{x\in [-1,1]} |\rho^{(n)}(x)|$. As $\lim_{x\to 0} \Psi(2 x)/x^n = 0$ for any $n\in\mathbb{N}$ the above estimate leads to $$\lim_{S_0 \ni x\to 0}{\Psi_\infty^{(n)}(x)}=0, \:\:\:\:\: \forall n\in \mathbb{N},$$ and by \eqref{eqn:thisfunctionissmooth} Claim 1 follows. 
	
	Let $\O:=\{(z_1,z_2)\in\C^2: \text{Re}(z_2) > \Psi_\infty( \text{Re} (z_1))\}$ and $p:=(0,0)$. Claim 1 implies that $\O$ is a convex domain containing no affine complex lines with smooth boundary. 
	
	\textit{Claim 2.} $p\in\partial \O$ is a weakly Goldilocks point. 
	
Since $M_{\O,p}(r) = \Psi^{-1}_{\infty}(r) \leq \Psi^{-1}_0(r) \leq 1/ \log^2(1/r)$ we have that $p$ is a weakly Goldilocks point. 
	
	\textit{Claim 3.} $p\in\partial \O$ is a non-Goldilocks point.
	
	Let $U$ be any neighbourhood of $p$. Keeping the notation above, let $j$ be large enough so that, the complex line $\Re z_2 = L_{2j}(\Re z_1)$ has a non-empty relatively open intersection with $U\cap\partial\O$. Take $p'$ in the relative interior of the intersection, denote the inner unit normal to $\partial\O$ at $p'$ by $\eta_{p'}$ and let $X:=(1,(\Psi_0(t_{2j})-\Psi_0(t_{2j+1}))/(t_{2j}-t_{2j+1})).$ Then, by a direct calculation we have $$\lim_{r \to 0} \delta_\O(p'+ r \eta_{p'}; X/\|X\|)>0.$$ As $U$ was arbitrary, we conclude that $p\notin \partial_{lg} \O$. 
	
	By Claims 2 and 3, we are done.
\end{proof}

\begin{rem}\label{rem:convexample2}
	There exists a convex domain $\O$ with smooth boundary such that the set of non-visible points in $\partial\O$ is not a closed set. 
\end{rem}

\begin{proof}
Keeping the notation in the proof of Remark \ref{rem:convexample} consider the real hyperplanes
given by $\Re z_2 = L_j(\Re z_1)$. Set $$\Psi_0(\Re z_1,\Im z_1, \Im z_2):=\max\{\Psi(\sqrt{\|z_1\|^2+|\Im(z_2)|^2}), L_n(\Re z_1):n\in 2\mathbb{N}\}.$$
	
	By taking appropriate convolutions as in the proof of Remark \ref{rem:convexample}, one may find a smooth convex function $\Psi_\infty$ that agrees with $\Psi_0$ outside of small neighborhoods of the disjoint sets $\partial E_{2j}$, where
	\[
	E_{2j}:= \left\{ 
	(z_1, \Im z_2) \in \mathbb{C}\times \mathbb{R}:
	\Psi(\sqrt{\|z_1\|^2+|\Im(z_2)|^2}) \le 
	L_{2j}(\Re z_1)
	\right\},
	\]
and so $\Psi_\infty$ agrees with each $L_{2j}$ 	in mutually disjoint neigbourhoods of $((1/2)(t_{2j}+t_{2j+1}),0)\in \C \times \mathbb R$. Then the domain $\O:=\{(z_1,z_2)\in\C^2: \Re (z_2) > \Psi_\infty(z)\}$ is a convex domain containing no affine complex lines, with smooth boundary. The domain $\O$ fails the visibility property as it contains non-trivial disjoint open complex faces denoted by $F_j \subset \{\Re z_1 >0\}$. In fact, for any $j$, one may find non-visible pairs $\{p_j,q_j\}\in F_j$. Since $p_j\to p:=(0,0)$, $p$ is a limit of non-visible points.
	Note that $\partial\O$ consists of strongly pseudoconvex points away from $F:=\{p\}\cup\left(\cup_{n=0}^\infty F_n\right)$. Furthermore,
$F_j\subset \{(z_1, \Psi_\infty(z_1,y_2)+iy_2): (z_1,y_2) \in E_{2j}\}$. But it is easy to see that $E_{2j}\subset [t_{2j+1},t_{2j}]\times i\mathbb R\times \mathbb R$, so that for $
(x_1+iy_1,y_2) \in E_{2j}$,
\[
 \frac12(t_{2j+2}+t_{2j+1}) <
t_{2j+1} \le x_1 \le \sqrt{x_1^2+y_1^2+y_2^2}
\le \Psi^{-1}(L_{2j}(x_1)) \le t_{2j} < \frac12(t_{2j-1}+t_{2j}).
\]

	Take $q\in\partial\O$ with $q\neq p$ then observe that by above any curve joining points tending to $p$ and $q$ respectively must have a part crossing the ``sphere'' 
	\[
	\{ (z_1, \Psi_\infty((z_1,y_2)+iy_2): \|(z_1,y_2)\| = \frac12(t_{2j+1}+t_{2j})\}, 
	\]
	and so
	tending to $\partial\O\setminus F$. Then, the proof of \cite[Theorem 1.4]{BZ2} shows that $p$ is a visible point, so the set of non-visible points in $\partial\O$ is not closed.
\end{proof}

Note that the set of non-Goldilocks points of $\partial\O$ is closed by definition. Then the example above shows that failure of the local Goldilocks point conditions cannot describe the set of non-visible points, as the set of non-visible points is not necessarily closed.

\qquad

\noindent	\textbf{Acknowledgements.} The author is grateful to his advisor Pascal Thomas for introducing him to the problem, and for his valuable guidance and endless patience throughout this project. The author is thankful to his coadvisor Nikolai Nikolov for his support and many suggestions. The author thanks to Filippo Bracci for informing him about one of their results which made possible to include Subsection \ref{sec:new} in the second version of this manuscript. The author is also thankful to the referee for a careful reading of the manuscript and valuable remarks. 
\kern-3em
{}

\end{document}